\documentclass[reqno]{amsart}
\usepackage[utf8]{inputenc}
\usepackage{pmboxdraw}
\usepackage{amsmath}
\usepackage{amssymb}
\usepackage{xcolor}
\usepackage{graphicx}
\usepackage{float}
\usepackage{enumitem}
\usepackage[hidelinks]{hyperref}
\usepackage{cite}
\newtheorem{theorem}{Theorem}

\newtheorem{definition}{Definition}

\newcommand{\C}{\mathbb{C}}

\newcommand{\R}{\mathbb{R}}

\def\Det{\mathrm{det}}

\newcommand*{\diff}[1]{\mathrm{d} #1}

\newcommand{\dt}{\mathrm{d}t}

\newcommand{\dx}{\mathrm{d}x}

\newcommand{\ler}[1]{\left( #1 \right)}

\newcommand{\lers}[1]{\left\{ #1 \right\}}

\newcommand{\abs}[1]{\left| #1 \right|}

\newcommand{\norm}[1]{\left|\left|#1\right|\right|}

\newcommand{\be}{\begin{equation}}
\newcommand{\ee}{\end{equation}}
\newcommand{\ba}{\begin{array}}
\newcommand{\ea}{\end{array}}

\newcommand{\bec}{\begin{equation*}}
\newcommand{\eec}{\end{equation*}}

\newcommand\bA{\mathbf A}
\newcommand\bw{\mathbf w}

\DeclareMathOperator*{\argmin}{arg\,min}

\newcommand{\fel}{\frac{1}{2}}
\newcommand{\cH}{\mathcal{H}}
\newcommand{\cB}{\mathcal{B}}

\newcommand{\cP}{\mathcal{P}}

\newcommand\cN{\mathcal N}

\newcommand{\cW}{\mathcal{W}}

\newcommand{\tr}{\mathrm{tr}}
\newcommand{\dd}{\mathrm{d}}

\newcommand{\cA}{\mathcal{A}}
\renewcommand\l{\lambda}

\newcommand{\D}{\mathbf{D}}
\newcommand{\ben}{\begin{enumerate}}
\newcommand{\een}{\end{enumerate}}

\newcommand{\dwts}[2]{d_{\cW_2}^2(#1,#2)}

\newcommand\bh{\cB(\cH)}

\begin{document}

\title{Operator means, barycenters, and fixed point equations}

\author[D\'aniel Virosztek]{D\'aniel Virosztek}
\address{D\'aniel Virosztek, HUN-REN Alfr\'ed R\'enyi Institute of Mathematics\\ Re\'altanoda u. 13-15.\\Budapest H-1053\\ Hungary}
\email{virosztek.daniel@renyi.hu}

\date{}



\thanks{D. Virosztek is supported by the Momentum Program of the Hungarian Academy of Sciences under grant agreement no. LP2021-15/2021, and partially supported by the ERC Synergy Grant No. 810115.}

\begin{abstract}
The seminal work of \emph{Kubo} and \emph{Ando} \cite{kubo-ando} provided us with an axiomatic approach to means of positive operators. As most of their axioms are algebraic in nature, this approach has a clear algebraic flavor. On the other hand, it is highly natural to take the geometric viewpoint and consider a distance (understood in a broad sense) on the cone of positive operators, and define the mean of positive operators by an appropriate notion of the center of mass. This strategy often leads to a fixed point equation that characterizes the mean. The aim of this survey is to highlight those cases where the algebraic and the geometric approaches meet each other. 
\end{abstract}

\maketitle

\tableofcontents

\section{An algebraic approach to operator menas} \label{sec:alg-approach}

Although the \emph{paralell sum} $A : B =\ler{A^{-1}+B^{-1}}^{-1}$ of positive definite matrice was considered already in 1969 by \emph{Anderson} and \emph{Duffin} \cite{anderson-duffin}, who used this operation to study electrical networks, and the \emph{geometric mean} $A \# B= A^{\fel}\ler{A^{-\fel} B A^{-\fel}}^{\fel} A^{\fel}=B^{\fel}\ler{B^{-\fel} A B^{-\fel}}^{\fel} B^{\fel}$ was introduced by \emph{Pusz} and \emph{Woronowicz} in 1975 \cite{pusz-woronowicz}, the first systematic study of means of positive operators is due to \emph{Kubo} and \emph{Ando} from 1980 \cite{kubo-ando}, who provided an axiomatic approach to operator means which proved extremely influential in the last decades. The binary operations characterized by their axioms are now called \emph{Kubo-Ando} connections and means, and the study of these means is flourishing research field within operator theory with intimate connections to quantum information theory. 
\par
This first section is very brief introduction to the Kubo-Ando theory which, however, aims to cover the most important concepts and phenomena. First, we need to fix some notation. Let $\cH$ be a complex Hilbert space, either finite or infinite dimensional, and let $\bh$ denote the set of all bounded linear operators on $\cH.$ Let $\bh^{sa}, \, \bh^{+},$ and $\bh^{++}$ stand for the set of all bounded self-adjoint, positive semidefinite, and positive definite (i.e., positive semidefinite and invertible) operators, respectively. The symbol $I$ stands for the identity of $\bh,$ and we consider the \emph{L\"owner order} induced by positivity on $\bh^{sa},$ that is, by $A \leq B$ we mean that $B-A$ is positive semidefinite, and $A<B$ means that $B-A$ is positive definite. The spectrum of $X \in \bh$ is denoted by $\mathrm{spec}(X).$ The symbols $\D$ and $\D^2$ denote the first and second Fr\'echet derivatives, respectively.

\begin{definition}[Operator connections and means]
A binary operation 
$$
\sigma: \bh^{+} \times \bh^{+} \rightarrow \bh^{+}; \, (A,B) \mapsto A \sigma B
$$
is called an \emph{operator connection} if it satisfies the following three properties:
\begin{itemize}
    \item (P1) Monotonocity in both variables: 
    $$
    \text{if } A \leq A' \text{ and } B \leq B' \text{ then } A \sigma B \leq A' \sigma B'
    $$
    \item (P2) Transformer inequality: 
    $$
    C \ler{A \sigma B} C \leq (C A C) \sigma (C B C) \text{ for all }  A,B,C \in \bh^{+} 
    $$
    \item (P3) Continuity for decreasing sequences: if $A_1\geq A_2 \geq A_3 \geq \dots$ and $A_n \rightarrow A$ in the strong operator topology, and similarly, $B_1\geq B_2 \geq B_3 \geq \dots$ and $B_n \rightarrow B$ strongly, then
    $$
    \ler{A_n} \sigma \ler{B_n} \rightarrow A \sigma B \text{ in the strong operator topology}
    $$
\end{itemize}
Furthermore, \emph{operator means} are those operator connections that satisfy the 
\begin{itemize} 
    \item (P4) Normalization condition: $I \sigma I=I$
\end{itemize}
\end{definition}
One of the most striking results of \cite{kubo-ando} is that there is one-by-one correspondence between operator connections and operator monotone functions mapping the positive half-line into itself. Namely, for every operator connection $\sigma$ there is a unique operator monotone function $f: (0, \infty) \rightarrow (0,\infty)$ such that 
\be \label{eq:conn-funct-corres}
A \sigma B = A^{\fel} f \ler{A^{-\fel} B A^{-\fel}} A^{\fel} \qquad (A, B \in \bh^{++}).
\ee
The $\sigma \leftrightarrow f$ correspondence described in \eqref{eq:conn-funct-corres} is in fact an affine order-isomorphism. It is instructive to take a look at those particular connections that have already been mentioned, and note that the parallel sum $A:B=\ler{A^{-1}+B^{-1}}^{-1}$ corresponds to the operator monotone function $f(x)=\frac{x}{x+1}$ while the geometric mean $A \# B=A^{\fel}\ler{A^{-\fel} B A^{-\fel}}^{\fel} A^{\fel}$ corresponds to $f(x)=\sqrt{x}.$
\par
Operator monotone functions mapping the positive half-line $(0, \infty)$ into itself admit a transparent integral-representation by L\"owner's theory.
In \cite{kubo-ando}, the following integral representation was considered:
\be \label{eq:int-rep-orig}
f(x)=\int_{[0, \infty]}\frac{x(1+t)}{x+t} \dd m(t) \qquad \ler{x>0},
\ee
where $m$ is a positive Radon measure on the extended half-line $[0, \infty].$ There is another integral-representation formula for operator monotone functions which may be even more convenient than \eqref{eq:int-rep-orig}. By a simple push-forward of $m$ by the transformation $T: [0, \infty] \rightarrow [0,1]; \, t \mapsto \l:=\frac{t}{t+1},$ we get the following integral-representation of positive operator monotone functions on $(0, \infty):$
\be \label{eq:int-rep-new}
f_{\mu}(x)=\int_{[0, 1]}\frac{x}{(1-\l)x+\l} \dd \mu(\l) \qquad \ler{x>0},
\ee
where $\mu=T_{\#} m,$ that is, $\mu(A)=m\ler{T^{-1}(A)}$ for every Borel set $A \subseteq [0,1].$ This representation is also well-known and appears --- among others --- in \cite{hansen-laa-13}.
Note that if $m$ is absolutely continuous with respect to the Lebesgue measure and $\dd m(t)=\rho(t) \dd t,$ then the density of $\mu=T_{\#} m$ is given by $\dd \mu(\l)=\frac{1}{(1-\l)^2}\rho\ler{\frac{\l}{1-\l}} \dd \l.$
\par
In particular, by \eqref{eq:conn-funct-corres} and \eqref{eq:int-rep-new}, there is an affine isomorphism between operator connections and positive Radon measures on $[0,1].$ We note that the normalization condition $I \sigma I=I$ is satisfied if and only if $\mu([0,1])=1.$
\par
We denote by $\cP\ler{[0,1]}$ the set of all Borel probability measures on $[0,1],$ and by $c\ler{\mu}:=\int_{[0,1]}\l \dd \mu (\l)$ the center of mass of $\mu.$ There is a natural way to assign a weight parameter to a mean $\sigma,$ namely, $W\ler{\sigma}:=f'(1)=c\ler{\mu},$ where $f$ generates $\sigma$ in the sense of \eqref{eq:conn-funct-corres} and $\mu$ generates $f$ in the sense of \eqref{eq:int-rep-new}. This weight parameter will play an essential role later when we turn to the discussion of \emph{generalized Hellinger distances}. Here we only mention that for the weighted arithmetic, geometric, and harmonic means generated by the functions
$$
a_\l(x)=(1-\l)+\l x, \, g_\l(x)=x^\l, \text{ and } h_\l(x)=\ler{(1-\l) +\l x^{-1}}^{-1},
$$
respectively, we have
$W\ler{\sigma_{a_\l}}=W\ler{\sigma_{g_\l}}=W\ler{\sigma_{h_\l}}=\l.$ That is, this weight parameter coincides with the usual one in the most important special cases.
\par
The \emph{convex order} is a well-known relation between probability measures; for $\mu, \nu \in \cP\ler{[0,1]},$ we say that $\mu \preccurlyeq \nu$ if for all convex functions $u: [0,1] \rightarrow \R$ we have $\int_{[0,1]} u \, \dd \mu \leq \int_{[0,1]} u \, \dd \nu.$ It is clear that for all $\mu \in \cP\ler{[0,1]}$ with $c\ler{\mu}=\l$ we have $\delta_\l \preccurlyeq \mu \preccurlyeq (1-\l)\delta_0 + \l \delta_1,$ where $\delta_x$ denotes the Dirac mass concentrated on $x.$ For any fixed $x>0,$ the map $\l \mapsto \frac{x}{(1-\l)x+\l}$ --- which is the core of the integral representation \eqref{eq:int-rep-new} --- is convex. Therefore, if $\mu \preccurlyeq \nu,$ then $f_\mu(x)\leq f_\nu(x)$ for all $x>0,$ and hence $A \sigma_{\mu} B \leq A \sigma_{\nu} B$ for all $A,B \in \bh^{++},$ where $f_\mu$ denotes the generating function corresponding to $\mu$ via \eqref{eq:int-rep-new} and $\sigma_\mu$ denotes the operator mean corresponding to $f_\mu$ via \eqref{eq:conn-funct-corres}. 
Consequently, if $\nu=\ler{1-c\ler{\mu}}\delta_0 + c\ler{\mu} \delta_1,$ then
$A \sigma_{\nu} B-A \sigma_{\mu} B$ is always positive, in particular, $\tr \ler{A \sigma_{\nu} B-A \sigma_{\mu} B}\geq 0.$ This quantity is exactly the generalized quantum Hellinger divergence we will discuss in Section \ref{susec:q-hell-div}.
\par
Let us turn to the description of two interesting involutive operations on operator connections. The \emph{adjoint} of a conncetion $\sigma$ is denoted by $\sigma^*$ and is defined by
\be \label{eq:conn-adjoint-def}
A \sigma^* B = \ler{A^{-1} \sigma B^{-1}}^{-1}
\ee
for $A,B \in \bh^{++}.$ If $f$ is the generating function of $\sigma$ in the sense of \eqref{eq:conn-funct-corres}, then $\sigma^*$ is generated by $f^*(x)=1/f(1/x).$ Note that taking the inverse on positive definite operators is an order-reversing operation, and hence $f^*$ is operator monotone whenever $f$ is so. Important examples of self-adjoint connections are the \emph{weighted geometric means} generated by the power functions $[0, \infty) \ni x \mapsto x^p$ for $0\leq p \leq 1.$ 
The \emph{transpose} of $\sigma$ is denoted by $\sigma^t$ and is defined by
\be \label{eq:conn-trans-def}
A \sigma^t B = B \sigma A \qquad (A, B \in \bh^{+}).
\ee
An important fact is that if $\sigma$ is represented by $f,$ then its transpose $\sigma^t$ is represented by $x f(1/x).$ Therefore, a connection $\sigma$ is symmetric if and only if its generator $f$ satisfies $f(x)=x f(1/x).$
\par
All in all, the Kubo-Ando theory is a beautiful and satisfactory theory of two-variable operator means. However, it leaves the problem multivariate operator means untouched. A clear advantage of the geometric approach to be presented in the next section is that it produces natural candidates for means of several positive operators. Apparently, a substantial part of the studies of Riemannian geometries on positive operators was motivated by the problem of finding appropriate multivariate counterparts of well-established bi-variate operator means. 
\section{A geometric approach to operator means} \label{sec:geom-approach}

The notion of barycenter --- or least squares mean --- plays a central role in averaging procedures related to various topics in mathematics and mathematical physics. Given a metric space $\ler{X, d}$ and an $m$-tuple $a_1, \dots, a_m$ in $X$ with positive weights $w_1, \dots, w_m$ such that $\sum_{j=1}^m w_j=1,$ the barycenter is defined to be
\be \label{eq:metric-barycenter-def}
\argmin_{x \in X}\sum_{j=1}^m w_j d^2\ler{a_j,x}.
\ee
The object defined above by \eqref{eq:metric-barycenter-def} plays an important role in various areas of mathematics, and hence can be found under various names: it is sometimes called \emph{Fr\'echet mean} or \emph{Karcher mean} or \emph{Cartan mean}. It is also called \emph{mean squared error estimator} for the following reason: imagine that we want to determine an unknown element of the metric space $\ler{X,d}$ and we can perform imperfect measurements many times. If we measure $a_j$ with relative frequency $w_j,$ then an attractive estimation of the unknown object is \eqref{eq:metric-barycenter-def} which minimizes the weighted mean squared error from the measured objects.
\par
In the sequel, we review two distinguished metrics on the cone of positive definite operators on finite dimensional Hilbert spaces, namely the \emph{Riemannian trace metric} and the \emph{Bures-Wasserstein metric.} The focus of these review will be on the barycenters determined by these metrics.

\subsection{The Riemannian trace metric} \label{susec:rtm}

The convex \emph{Boltzmann entropy} (or H-functional) of a random variable $X$ with probability density $\varrho$ is given by
\be \label{eq:entropy-def}
H(X)=\int_{\mathrm{supp}(X)} \varrho(x) \log \varrho(x) \diff{x}.
\ee
This is a particularly important functional; for instance, the heat equation
$$
\partial_t \varrho = \Delta \varrho
$$
can be seen as the gradient flow for the Boltzmann entropy as potential (or "energy") in the differential structure induced by optimal transportation \cite{JKO-98,JKO-97a,JKO-97b}.
\par
Let us restrict our attention to special random variables. A circularly-symmetric centered complex Gaussian distribution on $\C^N$ is completely described by its covariance matrix $\Sigma.$ The probability density of such a zero-mean complex Gaussian $Z \sim \cN_{\C}\ler{0, \Sigma}$ is given by
$$
f_{\cN_{\C}\ler{0, \Sigma}}(z)=\frac{\exp{\ler{-z^* \Sigma^{-1} z}}}{ \pi^N \Det{\Sigma}}.
$$
Consequently, the Boltzmann entropy of $Z \sim \cN_{\C}\ler{0, \Sigma}$ can be given in the following simple closed form:
\be \label{eq:entropy-of-normal}
H(Z)=-\tr \log \Sigma + C(N)
\ee
where $C(N)$ is an irrelevant constant depending only on the dimension $N.$ So the Boltzmann entropy is a smooth convex functional on the sub-manifold of centered, circularly-symmetric, non-degenerate complex Gaussians. We identify this sub-manifold with the cone of positive definite $N \times N$ matrices by the convenient identification of the random variable with its covariance matrix.
\par
Direct computations shows that the Hessian (that is, the second derivative) of the Boltzmann entropy \eqref{eq:entropy-of-normal} on these appropriate Gaussians is given by
\be \label{eq:ent-hess}
\D^2 H (A)[Y,X]= \tr A^{-1} Y A^{-1}X.
\ee
This is a collection of positive definite bilinear forms on the tangent spaces 
$$
T_A \ler{M_N^{++}(\C)} \simeq M_N^{sa}(\C)
$$
that depends smoothly on the foot point $A.$ Therefore, \eqref{eq:ent-hess} is a Riemannian metric tensor field. The global metric induced by this Riemannian tensor field
$$
g_A(X,Y):=\tr A^{-1} Y A^{-1}X
$$
is called the \emph{Riemannian trace metric (RTM)}.
\par
A particularly nice feature of the Riemannian trace metric is that the geodesic curve connecting the points $A,B \in \bh^{++}$ has the following simple closed form:
\be \label{eq:rtm-geod} 
\gamma_{A\rightarrow B}(t)=A^{\fel}\ler{A^{-\fel}BA^{-\fel}}^t A^{\fel}.
\ee
That is, the geodesic consists of the \emph{weighted geometric means} where the weight parameter $t$ runs from $0$ to $1.$ Consequently, the derivative of the geodesic reads as 
\be \label{eq:rtm-geod-velocity}
\gamma'_{A\rightarrow B}(t)=A^{\fel}\ler{A^{-\fel}BA^{-\fel}}^t \log \ler{A^{-\fel}BA^{-\fel}} A^{\fel},
\ee
and the RTM has the following simple closed form:
$$
d_{RTM}(A,B)
=\int_0^1\sqrt{g_{\gamma_{A\rightarrow B}(t)}\ler{\gamma'_{A\rightarrow B}(t),\gamma'_{A\rightarrow B}(t)}}\diff{t}
$$
\be \label{eq:rtm-closed-form}
=\int_0^1\sqrt{\tr \ler{\ler{\gamma_{A\rightarrow B}(t)}^{-1}\gamma'_{A\rightarrow B}(t)}^2}\diff{t}=\norm{\log \ler{A^{-\fel}B A^{-\fel}}}_{2}.
\ee
The notation $\norm{\cdot}_2$ stands here and throughout this paper for the Hilbert-Schmidt norm $\norm{X}_2:=\sqrt{\tr \ler{X^*X}}.$
\par
Therefore, the barycenter \eqref{eq:metric-barycenter-def} of the positive definite operators $A_1, \dots, A_m \in \bh^{++}$ with probability weights $w_1, \dots w_m$ is 
\be \label{eq:barycenter-rtm}
\argmin_{X \in \bh^{++}}\sum_{j=1}^m w_j d_{RTM}^2\ler{A_j,X}
=\argmin_{X \in \bh^{++}}\sum_{j=1}^m w_j \norm{\log \ler{X^{-\fel}A_j X^{-\fel}}}_{2}^2.
\ee
One can use Karcher's formula \cite[Theorem 2.1]{karcher} to compute the gradient of the objective function
\be \label{eq:rtm-bary-obj}
X \mapsto \sum_{j=1}^m w_j \norm{\log \ler{X^{-\fel}A_j X^{-\fel}}}_{2}^2
\ee
and deduce that the barycenter \eqref{eq:barycenter-rtm} --- which is more often called \emph{Karcher mean} in this context --- coincides with the unique positive definite solution of the \emph{Karcher equation} 
\be \label{eq:karcher}
\sum_{j=1}^m w_j \log \ler{X^{\fel} A_j^{-1} X^{\fel}}=0.
\ee
See \cite{moakher} and \cite{bhatia-pos-def-book} for alternative approaches on the derivation of the Karcher equation \eqref{eq:karcher}. 
\subsection{The Bures-Wasserstein metric} \label{susec:bures-wasserstein}
The {\it classical optimal transport (OT) problem} is to arrange the transportation of goods from producers to consumers in an optimal way, given the distribution of the sources and the needs (described by probability measures $\mu$ and $\nu$), and the cost $c(x,y)$ of transporting a unit of goods from $x$ to $y.$ Accordingly, a {\it transport plan} is modeled by a probability distribution $\pi$ on the product of the initial and the target spaces, where $\dd \pi(x,y)$ is the amount of goods to be transferred from $x$ to $y,$ and hence the marginals of $\pi$ are $\mu$ and $\nu.$ So the \emph{optimal transport cost} is the minimum of a convex optimization problem with linear loss function:
\be \label{eq:class-ot-cost-def}
\mathrm{Cost}\ler{\mu,\nu,c}= \min \lers{ \iint_{X \times Y} c(x,y) \dd \pi\ler{x,y} \, \middle| \, (\pi)_1=\mu, \, (\pi)_2=\nu}
\ee
where $(\pi)_i$ denotes the $i$th marginal of $\pi,$ and $X$ is the initial and $Y$ is the target space.
\par
OT costs \eqref{eq:class-ot-cost-def} give rise to OT distances (Wasserstein distances) on measures for certain cost functions $c(.,.)$. A prominent example is the \emph{quadratic Wasserstein distance} between probabilities on $\R^n$ having finite second moment defined by
\be \label{eq:2-wass-def-class}
\dwts{\mu}{\nu}= \inf \lers{ \iint_{\R^n \times \R^n} \abs{x-y}^2 \dd \pi\ler{x,y} \, \middle| \, (\pi)_1=\mu, \, (\pi)_2=\nu}.
\ee
The above definition \eqref{eq:2-wass-def-class} of the $2$-Wasserstein metric is \emph{static in nature}. It refers only to the initial and final distributions of the mass to be transported. The \emph{dynamical theory of mass transport} concerns on the contrary flows of measures connecting the initial and final states. The optimization problem is \emph{minimizing the total kinetic energy} needed to perform the transport. More precisely, the task is to minimize the kinetic energy over flows of measures $\ler{\rho_t}_{t=0}^T$ connecting the initial and final distributions and time-depending velocity fields $\ler{v_t}_{t=0}^T$ governing the flows --- we say that the velocity field $\ler{v_t}_{t=0}^T$ governs the flow $\ler{\rho_t}_{t=0}^T$ if they satisfy the linear transport equation (or continuity equation)
\be \label{eq:linear-transport}
\frac{\partial \rho_t}{ \partial t} +\nabla_x \cdot \ler{\rho_t v_t}=0.
\ee
Accordingly, the formula for the minimal kinetic energy (MKE) needed to transform $\mu$ to $\nu,$ if the total time allocated for the transport is $T,$ reads as follows:
$$
\mathrm{MKE}_T\ler{\mu, \nu}=
$$
\be \label{eq:mke}
=\inf \lers{ \int_0^T \int_{\R^n} \rho_t(x)\abs{v_t(x)}^2 \dx \, \dt \, \middle| \, \frac{\partial \rho_t}{ \partial t} +\nabla_x \cdot \ler{\rho_t v_t}=0, \, \rho_0=\mu, \rho_T=\nu}.
\ee

A seminal result of Benamou and Brenier \cite{Ben-Bren00} connects the static and the dynamic theory beautifully: the static $2$-Wasserstein distance \eqref{eq:2-wass-def-class} and the minimal kinetic energy needed to perform the dynamics are essentially the same.
More precisely, the Benamou-Brenier formula \cite{Ben-Bren00} tells us that the minimizing flow of measures in \eqref{eq:mke} is given by the \emph{displacement interpolation}, and
\be \label{eq:bb-form}
\mathrm{MKE}_T\ler{\mu, \nu}=\frac{1}{T} \dwts{\mu}{\nu}.
\ee
Now, if we take another look at \eqref{eq:mke}, keeping \eqref{eq:bb-form} in mind, we may observe that the $2$-Wasserstein distance $d_{\cW_2}$ is given by a formula that looks very much like a Riemannian geodesic formula. And indeed, there is a Riemannian metric tensor field on the space of probabilities that gives rise to the $2$-Wasserstein metric. The discussion of this Riemannian metric in general is beyond the scope of this survey --- we will be interested only in the special case of centered Gaussian measures. We only mention that this line of research was pioneered by Otto \cite{otto-cpde-2001} and we refer the interested reader to Subsection 8.1.2. of \cite{Villani1} for a detailed description of the theory. We must note, however, the groundbreaking discovery of Jordan, Kinderlehrer, and Otto \cite{JKO-97a,JKO-97b,JKO-98} who proved that the heat flow is the gradient flow of the Boltzmann entropy with respect to the $2$-Wasserstein Riemannian metric on probability densities.
\par
We shall restrict our attention to non-degenerate centered Gaussian measures on $\C^n$ that we identify with their non-singular covariance matrices. If $\mu$ is the law of the random variable $X \sim \cN_{\C}\ler{0, A}$ and $\nu$ is the law of $Y \sim \cN_{\C}\ler{0, B},$ then the quadratic Wasserstein distance between $\mu$ and $\nu$ admits the following closed form that refers only to the covariance matrices \cite{agueh-carlier, takatsu}:
\be \label{eq:W2-dist-fact}
d_{\cW_2}(\mu,\nu)=\ler{\tr A +\tr B - 2 \tr \ler{A^{\fel} B A^{\fel}}^{\fel}}^{\fel}.
\ee
The distance between the positive definite operators $A,B \in \bh^{++}$ acting on the Hilbert space $\cH=\C^n$ that appears on the right-hand side of \eqref{eq:W2-dist-fact} has a quantum information theoretic interpretation, as well. In that context, the name \emph{Bures distance} is used more frequently. We refer the reader to \cite{bhat-jain-lim} for a thorough study of this \emph{Bures-Wasserstein} distance given for $A, B \in \bh^{++}$ by
\be \label{eq:bw-dist-def}
d_{BW}^2(A,B)=\tr A +\tr B - 2 \tr \ler{A^{\fel} B A^{\fel}}^{\fel}.
\ee
We note furthermore that the isometries of the density spaces of $C^*$-algebras with respect to \eqref{eq:bw-dist-def} have been determined by Moln\'ar in \cite{molnar-bures}.
The geodesic line segment in the Bures-Wasserstein metric connecting $A$ with $B$ has the following simple closed form \cite{bhat-jain-lim, bhat-jain-lim-wass-mean}:
\be \label{eq:BW-geod-form}
A \lozenge_t B =(1-t)^2 A^2 + t^2 B^2+t(1-t)\ler{(AB)^{\fel}+(BA)^{\fel}}
\ee
where $t$ runs from $0$ to $1,$ and the square roots of the non-Hermitian operators are understood as follows: $(AB)^{\fel}=A^{\fel}\ler{A^{\fel}BA^{\fel}}^{\fel}A^{-\fel}$ and 
$$
(BA)^{\fel}=B^{\fel}\ler{B^{\fel} A B^{\fel}}^{\fel}B^{-\fel}=A^{-\fel}\ler{A^{\fel}BA^{\fel}}^{\fel}A^{\fel}.
$$
The elements of the geodesic segment \eqref{eq:BW-geod-form} are Bures-Wasserstein barycenters with appropriate weights \cite{bhat-jain-lim-wass-mean}, namely,
\be \label{eq:BW-geod-bary}
A \lozenge_t B= \argmin_{X \in \bh^{++}} (1-t) d_{BW}^2(A,X)+t d_{BW}^2(B,X) \qquad (A, B \in \bh^{++}).
\ee
An attractive feature of the Bures-Wasserstein barycenter is that it is characterized by a rather simple fixed point equation: the unique minimizer of the functional
\be \label{eq:BW-quad-loss-fn}
X \mapsto \sum_j w_j d_{BW}^2 (A_j,X)
\ee
on the positive definite cone coincides with the unique positive definite solution of the operator equation 
\be \label{eq:BW-fixed-point}
X=\sum_{j=1}^m w_j \ler{X^{\fel} A_j X^{\fel}}^{\fel}.
\ee
This striking result was first proved in \cite{agueh-carlier} and a very transparent presentation of the proof can be found, e.g., in \cite[Section 6]{bhat-jain-lim}.

\subsection{Barycenters for generalized quantum Hellinger distances} \label{susec:q-hell-div}

As one can see in \eqref{eq:bw-dist-def}, the Bures-Wasserstein distance is the square root of the distance between the trace of the arithmetic mean of $A$ and $B$ and the trace of a certain geometric mean of the same operators. Therefore it is a non-commutative version of the \emph{Hellinger distance} of probability vectors defined by 
\be \label{eq:hell-dist-classic}
d_H^2\ler{(p_1, \dots, p_n), (q_1, \dots q_n)}=\sum_{j=1}^n \ler{\sqrt{p_j}-\sqrt{q_j}}^2.
\ee
A thorough study of a variety of quantum Hellinger distances is presented in \cite{bhatia-paper}. Somewhat later, \emph{generalized} quantum Hellinger distances were introduced and studied with a strong emphasis on the characterization of the barycenter \cite{pv-hellinger}. Very recently, a one-parameter family of distances including the Bures-Wasserstein distance and a certain Hellinger distance was proposed and studied in \cite{komalovics-molnar}.
\par
Given a Borel probability measure $\mu \in \cP([0,1]),$ the corresponding \emph{generalized quantum Hellinger divergence} is given by
\be \label{eq:q-hell-div-def-1}
\phi_\mu(A,B):=\tr \ler{\ler{1-c\ler{\mu}} A + c\ler{\mu} B - A \sigma_{f_\mu} B} \qquad \ler{A,B \in \bh^{++}},
\ee
where $c(\mu)=\int_{{0,1}} \lambda \dd \mu(\lambda)$ is the center of mass of $\mu,$ and $\sigma_{f_\mu}$ is the Kubo-Ando mean generated by the operator monotone function $f_\mu$ in the sense of \eqref{eq:conn-funct-corres} and $f_\mu$ is determined by $\mu$ in the sense of \eqref{eq:int-rep-new}. We believe that the characterization of the barycenter of finitely many positive operators by a fixed point equation is instructive, especially as we did not present the proof of the analogous result for the Bures-Wasserstein metric. The following result and its proof appeared originally in \cite{pv-hellinger}.

\begin{theorem} \label{thm:gen-hellinger-barycenter}
Let $\mu \in \cP{[0,1]}$ and let $\phi_\mu$ be the generalized quantum Hellinger divergence generated by $\mu$ given in \eqref{eq:q-hell-div-def-1}.
The barycenter of the positive definite operators $A_1, \dots, A_m$ with positive weights $w_1, \dots, w_m$ with respect to $\phi_\mu,$ i.e.,
\be \label{eq:opt-prob}
\argmin_{X \in \bh^{++}} \sum_{j=1}^m w_j \phi_\mu\ler{A_j, X}
\ee
coincides with the unique positive definite solution of the fixed point equation
\be \label{eq:barycenter-char}
X= \frac{1}{c\ler{\mu}}
\sum_{j=1}^m w_j \int_{[0,1]} \l \abs{(1-\l) A_j^{-1} X^{\fel}+ \l X^{-\fel}}^{-2} \dd \mu (\l)
\ee
where $|\cdot|$ stands for the absolute value of an operator, that is, $\abs{Z}=\ler{Z^*Z}^{\fel}.$
\end{theorem}

\begin{proof}
Assume that the positive definite operators $A_1, \dots, A_m$ and the weights $w_1, \dots w_m$ are given. By the strict concavity of $f_\mu,$ the function
$$
X \mapsto \phi_\mu \ler{A,X}=\tr \ler{\ler{1-c\ler{\mu}} A + c\ler{\mu} X - A^{\fel}f_\mu\ler{A^{-\fel} X A^{-\fel}}A^{\fel}}
$$
is strictly convex on $\bh^{++},$ see, e.g., \cite[2.10. Thm.]{carlen}.
Therefore, there is a unique solution $X_0$ of \eqref{eq:opt-prob}, and it is necessarily a critical point of the function $X \mapsto \sum_{j=1}^m w_j \phi_\mu\ler{A_j, X}.$ That is, it satisfies
\be \label{eq:diff-vanishes}
\D \ler{\sum_{j=1}^m w_j \phi_\mu\ler{A_j, \cdot}}(X_0)[Y]=0 \qquad \ler{Y \in \bh^{sa}}.
\ee

Easy computations give that
\be \label{eq:diff-transformed}
\D \ler{\sum_{j=1}^m w_j \phi_\mu\ler{A_j, \cdot}}(X)[Y]= c\ler{\mu} \tr Y -\sum_{j=1}^m w_j \tr \D F_{\mu, A_j}(X)[Y],
\ee
where for a positive definite operator $A,$ the map $F_{\mu,A}: \bh^{++} \rightarrow \bh^{++}$ is defined by
\be \label{eg:f-a-def}
F_{\mu,A}(X):=A \sigma_{f_\mu} X = A^{\fel}f_\mu\ler{A^{-\fel} X A^{-\fel}}A^{\fel}.
\ee

By differentiating \eqref{eq:int-rep-new}, we have
\be \label{eq:div-x-t-rep}
\D f_\mu(X) [Y]=\int_{[0,1]} \l \ler{(1-\l) X + \l I}^{-1} Y \ler{(1-\l) X + \l I}^{-1} \dd \mu (\l)
\ee
for $X \in \bh^{++}, \, Y \in \bh^{sa}.$
Consequently,
$$
\D F_{\mu, A_j}(X)[Y]
$$
$$
=\int_{[0,1]} \l A_j^\fel \ler{(1-\l) A_j^{-\fel} X A_j^{-\fel}+\l I }^{-1} A_j^{-\fel} Y^{\fel} \times 
$$
$$
\times Y^{\fel} A_j^{-\fel} \ler{(1-\l) A_j^{-\fel} X A_j^{-\fel}+\l I }^{-1} A_j^\fel \dd \mu (\l)=
$$
\be \label{eq:diff-concrete-form}
=\int_{[0,1]} \l \ler{(1-\l)X A_j^{-1}+ \l I}^{-1} Y \ler{(1-\l) A_j^{-1} X+ \l I}^{-1}  \dd \mu (\l).
\ee
By the linearity and the cyclic property of the trace, we get from \eqref{eq:diff-transformed} and \eqref{eq:diff-concrete-form} that \eqref{eq:diff-vanishes} is equivalent to
\be \label{eq:}
\tr \left[Y \ler{c\ler{\mu} I-\sum_{j=1}^m w_j \int_{[0,1]} \l \abs{(1-\l) A_j^{-1} X+ \l I}^{-2} \dd \mu (\l)}\right] = 0 \qquad \ler{Y \in \bh^{sa}}.
\ee 
This latter equation amounts to 
\be \label{eq:barycenter-necessary}
c\ler{\mu} I= \sum_{j=1}^m w_j \int_{[0,1]} \l \abs{(1-\l) A_j^{-1} X+ \l I}^{-2} \dd \mu (\l).
\ee
Multiplying by $\frac{1}{\sqrt{c(\mu)}}X^{1/2}$ from both left and right gives the desired operator equation \eqref{eq:barycenter-char}.
\end{proof}

\section{Connections between the algebraic and the geometric approaches} \label{sec:alg-geom-conn}

This section is devoted to the phenomenon when the algebraic and the geometric approach to operator means meet each other, that is, when Kubo-Ando means admit barycentric interpretation.

It is well known that special Kubo-Ando operator means, namely, the arithmetic and the geometric means admit divergence center interpretations. The arithmetic mean
$A \nabla B=(A+B)/2$ is clearly the barycenter for the Euclidean metric on positive operators:
$$
A \nabla B= \argmin_{X>0} \fel \ler{\tr (A-X)^2+\tr (B-X)^2}.
$$
A much more interesting fact is that the geometric mean $A\#B$ is the barycenter for the Riemannian trace metric $d_{RTM}(X,Y)=\norm{\log{\ler{X^{-\fel}Y X^{-\fel}}}}_2,$ that is,
\be \label{eq:geom-mean-rep}
A \# B= \argmin_{X>0} \fel \ler{d_{RTM}^2(A,X)+d_{RTM}^2(B,X)}.
\ee
The barycentric representation \eqref{eq:geom-mean-rep} of the bivariate geometric mean opened the gate for the \emph{definition} of the multivariate geometric mean as the barycenter with respect to the Riemannian trace metric. This definition was introduced by Moakher \cite{moakher} and Bhatia-Holbrook \cite{bhat-holb}. A recent result in \cite{pv-ka-bary} tells us that \emph{every} symmetric Kubo-Ando means admits a divergence center interpretation. Now we turn to the detailed explanation of this latter result, and we follow the presentation of \cite{pv-ka-bary}.
\par
Let $\sigma: \, \bh^{++} \times \bh^{++} \rightarrow \bh^{++}$ be a symmetric Kubo-Ando operator mean, and let $f_\sigma: \, (0, \infty) \rightarrow (0,\infty)$ be the operator monotone function representing $\sigma$ in the sense that
\be \label{eq:f-sig-rep}
A \sigma B= A^{\fel}f_\sigma \ler{A^{-\fel} B A^{-\fel}} A^{\fel}.
\ee
Clearly, $f_\sigma(1)=1,$ and the symmetry of $\sigma$ implies that $f_\sigma(x)=x f_\sigma\ler{\frac{1}{x}}$ for $x>0,$ and hence $f_\sigma'(1)=1/2$. 
We define
$$
g_\sigma: \, (0, \infty) \supseteq \mathrm{ran}\ler{f_\sigma} \rightarrow [0, \infty)
$$
by
\be\label{eq:g}
g_\sigma(x):=\int_1^x\left(1-\frac{1}{f_\sigma^{-1}(t)}\right)\dd t.
\ee
Obviously, $g_\sigma(1)=0$, $g_\sigma'(x)=1-\frac{1}{f_\sigma^{-1}(x)},$ and $g_\sigma'(1)=0$ as $f_\sigma(1)=1$. Since $f_\sigma$ is strictly monotone increasing, so is $g_\sigma'$, and hence $g_\sigma$ is strictly convex on its domain. 
Now we define the quantity
\be \label{eq:fi-sig}
\phi_\sigma (A,B):=\tr \ler{g_\sigma\left( A^{-1/2}BA^{-1/2}\right)},
\ee
for positive definite operators $A,B\in\cB(\cH)^{++}$ such that the spectrum of $A^{-1/2}BA^{-1/2}$ is contained in $\mathrm{ran}\ler{f_\sigma}.$ We define $\phi_\sigma(A,B):= +\infty$ if $\mathrm{spec} \ler{A^{-1/2}BA^{-1/2}} \nsubseteq \mathrm{ran}\ler{f_\sigma}.$
It will be important in the sequel that by \cite[2.10. Thm.]{carlen} the strict convexity of $g_\sigma$ implies that $X \mapsto \phi_\sigma(A,X)$ is strictly convex (whenever finite) for every $\sigma$ and $A.$ Now we are in the position to formalize the divergence center interpretation of symmetric Kubo-Ando means. The precise statement reads as follows.

\begin{theorem} \label{thm:bary}
For any $A,B\in\cB(\cH)^{++},$
\be \label{eq:bary}
\argmin_{X\in\cB(\cH)^{++}}\frac{1}{2}\left(\phi_\sigma (A,X)+\phi_\sigma(B,X)\right)=A\sigma B.
\ee
That is, $A\sigma B$ is a unique minimizer of the function $X\mapsto \frac{1}{2}\left(\phi_\sigma (A,X)+\phi_\sigma(B,X)\right)$ on $\cB(\cH)^{++}$.
\end{theorem}
\begin{proof}
By the strict convexity of $X\mapsto \frac{1}{2}\left(\phi_\sigma (A,X)+\phi_\sigma(B,X)\right)$ it is sufficient to show that $A \sigma B$ is a critical point, and therefore a unique minimizer. First we compute the derivative 
\bec
\left.\frac{\dd}{\dd t}\right|_{t=0}\phi_\sigma(A,X+tY)=\left.\frac{\dd}{\dd t}\right|_{t=0}\tr g_\sigma\left(A^{-1/2}XA^{-1/2}+tA^{-1/2}YA^{-1/2}\right)=
\eec
\be\label{eq:derivative}
\tr A^{-1/2}g_\sigma'\left(A^{-1/2}XA^{-1/2}\right)A^{-1/2}Y
\ee
for all $Y\in\cB(\cH)^{sa}$. Since $g_\sigma'(x)=1-(f_\sigma^{-1}(x))^{-1}$, we get
\be\label{eq:derivative2}
\left.\frac{\dd}{\dd t}\right|_{t=0}\frac{1}{2}\phi_\sigma(A,X+tY)=\frac{1}{2}\tr \left(A^{-1/2}\left(I-\left[f_\sigma^{-1}\left(A^{-1/2}XA^{-1/2}\right)\right]^{-1}\right)A^{-1/2}Y\right)
\ee
for all $Y\in\cB(\cH)^{sa}$. Substituting $X=A\sigma B=A^{1/2}f_\sigma(A^{-1/2}BA^{-1/2})A^{1/2}$ into the derivative above, the right hand side of (\ref{eq:derivative2}) becomes
\bec
\frac{1}{2}\tr \left(A^{-1/2}\left(I-\left[f_\sigma^{-1}\left(A^{-1/2}A^{1/2}f_\sigma(A^{-1/2}BA^{-1/2})A^{1/2}A^{-1/2}\right)\right]^{-1}\right)A^{-1/2}Y\right)=
\eec
\be\label{eq:der3}
\frac{1}{2}\tr(A^{-1}-B^{-1})Y.
\ee
Since the operator mean $\sigma$ is symmetric, that is 
$$
A\sigma B=B\sigma A=B^{1/2}f_\sigma\left(B^{-1/2}AB^{-1/2}\right)B^{1/2},
$$
a similar computation for the derivative 
$$
\left.\frac{\dd}{\dd t}\right|_{t=0}\frac{1}{2}\phi_\sigma(B,X+tY)$$ 
at $X=A\sigma B$ gives
\be\label{eq:der4}
\frac{1}{2}\tr(B^{-1}-A^{-1})Y
\ee
for all $Y\in\cB(\cH)^{sa}$. Using (\ref{eq:der3}) and (\ref{eq:der4}) we get for the derivative
\bec
\left.\left.\frac{\dd}{\dd t}\right|_{t=0}\left\lbrace\frac{1}{2}\phi_\sigma(A,X+tY)+\frac{1}{2}\phi_\sigma(B,X+tY)\right\rbrace\right|_{X=A\sigma B}
\eec
\bec
=\frac{1}{2}\tr(A^{-1}-B^{-1})Y+\frac{1}{2}\tr(B^{-1}-A^{-1})Y=0
\eec
for all $Y\in\cB(\cH)^{sa}$. So we obtained that $A \sigma B$ is a critical point and hence a unique minimizer of $X\mapsto \frac{1}{2}\left(\phi_\sigma (A,X)+\phi_\sigma(B,X)\right).$
\end{proof}

The above characterization of symmetric Kubo-Ando means as barycenters (Theorem \ref{thm:bary}) naturally leads to the idea of defining weighted and multivariate versions of Kubo-Ando means as minimizers of appropriate loss functions derived from the divergence $\phi_\sigma.$
Given a symmetric Kubo-Ando mean $\sigma,$ a finite set of positive definite operators $\bA=\lers{A_1, \dots, A_m} \subset \bh^{++},$ and a discrete probability distribution $\bw=\lers{w_1, \dots, w_m} \subset (0,1]$ with $\sum_{j=1}^m w_j=1$ we define the corresponding loss function $Q_{\sigma, \bA, \bw}: \, \bh^{++} \rightarrow [0, \infty]$ by
\be \label{eq:Q-def}
Q_{\sigma, \bA, \bw}(X):=\sum_{j=1}^m w_j \phi_\sigma \ler{A_j,X}
\ee
where $\phi_\sigma$ is defined by \eqref{eq:fi-sig}.
\par
However, in the weighted multivariate setting, when $\mathrm{ran}\ler{f_\sigma}$ is smaller than the whole positive half-line $(0, \infty),$ then some undesirable phenomena occur which are illustrated by the next example.
\par
Consider the arithmetic mean generated by $f_\nabla(x)=(1+x)/2$ with $\mathrm{ran}\ler{f_\nabla}=\ler{\fel, \infty}.$ Let $A_1, A_2 \in \bh^{++}$ satisfy $ A_1 < \frac{1}{3} A_2.$ In this case, for any $\alpha \in (0,1),$ the loss function $Q_{\nabla, \lers{A_1, A_2}, \lers{1-\alpha, \alpha}} (X)$ is finite only if $X>\fel A_2.$ So the barycenter of $A_1$ and $A_2$ with weights $\lers{1-\alpha,\alpha}$ is separated from $A_1$ for every $\alpha \in (0,1),$ even for values very close to $0.$
\par
To exclude such phenomena, from now on, we assume that the range of $f_\sigma$ is maximal, that is, $\mathrm{ran}\ler{f_\sigma}=(0, \infty),$ and hence
$g_\sigma(x)=\int_1^x\left(1-\frac{1}{f_\sigma^{-1}(t)}\right)\dd t$ is defined on the whole positive half-line $(0, \infty).$ Consequently, $\phi_\sigma$ is always finite, and hence so is $Q_{\sigma, \bA, \bw}$ on the whole positive definite cone $\bh^{++}.$

\begin{definition}\label{def:bary_alpha}
Let $\sigma: \, \bh^{++} \times \bh^{++} \rightarrow \bh^{++}$ be a symmetric Kubo-Ando operator mean such that $f_\sigma: \, (0, \infty) \rightarrow (0,\infty),$ which is the operator monotone function representing $\sigma$ in the sense of \eqref{eq:f-sig-rep}, is surjective. Let $g_\sigma$ be defined as in \eqref{eq:g}, and $\phi_\sigma$ be defined as in \eqref{eq:fi-sig}.
We call the optimizer
\be \label{eq:bary_alpha}
\mathbf{bc}\ler{\sigma,\bA, \bw}:=\argmin_{X\in\cB(\cH)^{++}}Q_{\sigma, \bA, \bw}
\ee
the weighted barycenter of the operators $\lers{A_1, \dots, A_m}$  with weights $\lers{w_1, \dots, w_m}$.
\end{definition}
By Theorem \ref{thm:bary}, this barycenter may be considered as a weighted multivariate version of Kubo-Ando means. To find the barycenter $\mathbf{bc}\ler{\sigma, \bA, \bw},$ we have to solve the critical point equation
\be \label{eq:crit-point}
\D Q_{\sigma, \bA, \bw}[X](\cdot)=0
\ee
for the strictly convex loss function $Q_{\sigma, \bA, \bw},$ where the symbol
$$
\D Q_{\sigma, \bA, \bw}[X](\cdot) \in \mathrm{Lin}\ler{\bh^{sa}, \R}
$$
stands for the Fr\'echet derivative of $Q_{\sigma, \bA, \bw}$ at the point $X \in \bh^{++}.$ For any $Y \in \bh^{sa}$ we have
$$
\D Q_{\sigma, \bA, \bw}[X](Y)
= \sum_{j=1}^m w_j \D \ler{\phi_\sigma \ler{A_j,\cdot}}[X](Y)
$$
$$
= \sum_{j=1}^m w_j \D \ler{\tr \ler{g_\sigma\ler{A_j^{-\fel} \cdot A_j^{-\fel}}}}[X](Y)
$$
$$
= \sum_{j=1}^m w_j \tr \ler{g_\sigma'\ler{A_j^{-\fel} X A_j^{-\fel}}A_j^{-\fel} Y A_j^{-\fel}}
$$
$$
= \tr \ler{\ler{\sum_{j=1}^m w_j A_j^{-\fel} g_\sigma'\ler{A_j^{-\fel} X A_j^{-\fel}}A_j^{-\fel}} Y }
$$
that is, the equation to be solved is
\be \label{eq:crit-point-2}
\sum_{j=1}^m w_j A_j^{-\fel} g_\sigma'\ler{A_j^{-\fel} X A_j^{-\fel}}A_j^{-\fel}=0 .
\ee
By the definition of $g_\sigma,$ see \eqref{eq:g}, $g_\sigma'(t)=1-\frac{1}{f_\sigma^{-1}(t)}$ for $t \in (0, \infty),$ and hence the critical point of the loss function $Q_{\sigma, \bA, \bw}$ is described by the equation
\be \label{eq:crit-point-3}
\sum_{j=1}^m w_j A_j^{-\fel} \ler{I-\ler{f_\sigma^{-1}\ler{A_j^{-\fel} X A_j^{-\fel}}}^{-1}}A_j^{-\fel}=0.
\ee
For $\sigma=\#$ the generating function is $f_{\#}(x)=\sqrt{x},$ and hence the inverse is $f_{\#}^{-1}(t)=t^2.$ In this case, the critical point equation \eqref{eq:crit-point-3} describing the barycenter $\mathbf{bc}\ler{\#, \bA, \bw}$ reads as follows:
\be \label{eq:geom-bary}
\sum_{j=1}^m w_j \ler{A_j^{-1}-X^{-1} A_j X^{-1}}=0.
\ee
Note that \eqref{eq:geom-bary} may be considered as a \emph{generalized Riccati equation}, and in the special case $m=2, w_1=w_2=\fel,$ the solution of \eqref{eq:geom-bary} is the symmetric geometric mean $A_1 \# A_2.$
\par
More generally, if $m=2, w_1=1-\alpha,$ and $w_2=\alpha,$ then \eqref{eq:geom-bary} has the following form:
$$
(1-\alpha)A_1^{-1}+\alpha A_2^{-1}=(1-\alpha)X^{-1} A_1 X^{-1}+\alpha X^{-1} A_2 X^{-1},
$$
or equivalently
\be\label{eq:Riccati_uj-0}
X\left[(1-\alpha)A_1^{-1}+\alpha A_2^{-1}\right]X=(1-\alpha)A_1+\alpha A_2.
\ee
Recall that for positive definite $A$ and $B$, the Riccati equation 
$$
XA^{-1}X=B
$$
has a unique positive definite solution, that is the geometric mean 
$$
A\#B=A^{1/2}(A^{-1/2}BA^{-1/2})^{1/2}A^{1/2}.
$$
We can observe that \eqref{eq:Riccati_uj-0} is the Riccati equation for the weighted harmonic mean
$$
A_1 !_{\alpha} A_2=[(1-\alpha)A_1^{-1}+\alpha A_2^{-1}]^{-1}
$$
and the weighted arithmetic mean $A_1\nabla_{\alpha}A_2=(1-\alpha)A_1+\alpha A_2$, ie
\be\label{eq:Ric-0}
X(A_1 !_{\alpha} A_2)^{-1}X=A_1\nabla_{\alpha}A_2.
\ee
Hence the solution of \eqref{eq:Riccati_uj-0} is the geometric mean of the weighted harmonic and the weighted arithmetic mean
\be\label{eq:wgeom_uj-0}
X=(A_1 !_{\alpha} A_2)\#(A_1\nabla_{\alpha}A_2).
\ee
It means that in this case the weighted barycenter with respect to $\phi_{\#}$ does not coincide with the weighted geometric mean, nevertheless
$$
\mathbf{bc}\ler{\#, \lers{A_1, A_2}, \lers{1-\alpha, \alpha}}=(A_1 !_{\alpha} A_2)\#(A_1 \nabla_{\alpha} A_2)
$$
that is, $\mathbf{bc}\ler{\#, \lers{A_1, A_2}, \lers{1-\alpha, \alpha}}$ is the Kubo-Ando mean of $A_1$ and $A_2$ with representing function 
$$
f_{\ler{!_\alpha \# \nabla_\alpha}}(x)=
\sqrt{\frac{x(1-\alpha +\alpha x)}{(1-\alpha)x+\alpha}}.
$$
These means were widely investigated in \cite{kim-lawson-lim}.
\par
We note that the critical point equation \eqref{eq:geom-bary} can be rearranged as
\be \label{eq:geom-bary-v2}
X \ler{\sum_{j=1}^m w_j A_j^{-1}} X= \sum_{j=1}^m w_j A_j.
\ee
This is the Ricatti equation for the weighted multivariate harmonic mean $\ler{\sum_{j=1}^m w_j A_j^{-1}}^{-1}$ and arithmetic mean $\sum_{j=1}^m w_j A_j,$ hence the barycenter $\mathbf{bc}\ler{\#, \bA, \bw}$ coincides with the weighted $\cA \# \cH$-mean of Kim, Lawson, and Lim \cite{kim-lawson-lim}, that is,
\be \label{eq:bary-ah}
\mathbf{bc}\ler{\#, \bA, \bw}=\ler{\sum_{j=1}^m w_j A_j^{-1}}^{-1} \# \ler{\sum_{j=1}^m w_j A_j}.
\ee

\paragraph*{{\bf Acknowledgement}} I am grateful to the anonymous reviewer for his/her valuable comments and recommendations.

\paragraph*{{\bf Conflict of interest statement}} On behalf of all authors, the corresponding author states that there is no conflict of interest.

\begin{small}
\bibliographystyle{plainurl}  
\bibliography{surv-refs}
\end{small}

\end{document}